\documentclass[12pt]{amsart}
\usepackage{enumerate}
\usepackage{amssymb}
\usepackage{graphicx}
\usepackage{amscd, color}
\usepackage{amsmath}
\usepackage{amsfonts}
\usepackage[english]{babel}
\usepackage{epsfig}
\usepackage{xypic}
\usepackage{color}

\newtheorem{theorem}{Theorem}

\newtheorem{corollary}[theorem]{Corollary}
\newtheorem{example}[theorem]{Example}

\newtheorem{remark}[theorem]{Remark}

\newtheorem{lemma}[theorem]{Lemma}

\begin{document}

\title[Compact convex sets in $2$-D asymmetric normed lattices]{Compact convex sets in $2$-dimensional asymmetric normed lattices}

\author{N. Jonard-P\'erez and  E.A. S\'anchez-P\'erez}



\thanks{The first author has been supported by CONACYT (Mexico) under grant 204028.
The second author has been supported by the Ministerio de Econom\'{\i}a y Competitividad (Spain) under grant
 MTM2012-36740-C02-02.}

\maketitle

\begin{abstract} In this note,
we study the geometric structure of  compact convex  sets in $2$-dimensional asymmetric normed lattices.  We prove that every $q$-compact convex set is strongly $q$-compact and we give a complete geometric description of the compact convex set with non empty interior in $(\mathbb R^{2}, q)$, where $q$ is an asymmetric lattice norm.
\end{abstract}

\section{Introduction}

A  direct consequence of the Heine-Borel Theorem is that convex compact sets of the one dimensional Hausdorff linear spaces are the closed bounded intervals.
However, compact sets of the one-dimensional, non-Hausdorff, real asymmetric normed linear space $(\mathbb R, | \cdot|_a)$  have a completely different structure. The asymmetric norm $| \cdot |_a$ is defined by
$$
| r |_a := \max\{r,0\}, \quad r \in \mathbb R.
$$
This asymmetric norm  induces a topology on $\mathbb R$ in which the basic open sets are $B_\varepsilon(r)=\{y \in \mathbb R: y-r < \varepsilon \}$. It can be easily seen that in this space each convex compact set can be written either as an interval
$(a,b]$ or as an interval $[a,b]$, for $- \infty \le a \le b < \infty$ (see Corollary 4.3 in \cite{new}).

In the symmetric case, it is well-known that any two $n$-dimensional compact convex bodies are homeomorphic. Furthermore, any homeomorphism between two compact convex bodies $A$ and $B$ of the same dimension maps the interior of $A$ onto the interior of $B$ and the boundary of $A$ onto the boundary of $B$.

On the other hand, it is well known that every closed convex subset $K$ of a finite dimensional Banach space is homeomorphic to either  $[0,1]^{n}\times[0,\infty)$  or $ [0,1]^{n}\times \mathbb{R}^{m}$ for some  $n,m \in \mathbb{N}$ (see \cite{bespel}). In particular, if $K$ is unbounded and  has no lines, then it is homeomorphic to $[0,1]^{n}\times[0,\infty)$.
In the special case of $\mathbb R^{2}$ we have the following theorem (see \cite[Chapter III, \S 6]{bespel}):

\begin{theorem}\label{t:convexos dimension 2}
Let $K_1$ and $K_2$ be two closed convex subsets of $\mathbb R^{2}$.
Additionally, suppose that both sets have non-empty interior and none of them  contains  a line. If $K_1$ and $K_2$ are unbounded, then there exists an homeomorphism $f:\mathbb R^{2}\to \mathbb R^{2}$ sending $K_1$ onto $K_2$, the interior of $K_1$ onto the interior of $K_2$, and the boundary of $K_1$ onto the boundary of $K_2$.
\end{theorem}

Although some effort has been made to find an asymmetric version of the previous results (see \cite{todos,new,luis}), at the moment there is no explicit description in the mathematical literature of the (compact) convex sets of an asymmetric normed space. The aim of this work is to provide such a description in the two dimensional case.

After Corollary 11 in \cite{luis}, we know that only asymmetric normed spaces not satisfying  any separation axiom stronger  than $T_0$ provides a scenario that is different from  the classical one for finite dimensional normed spaces. This result establishes that a finite dimensional asymmetric normed space is normable if it is $T_1$, and so in this case the results on representation of compact convex sets are the same as those in normed spaces.

 As in \cite{new},  we restrict our attention to the case of asymmetric norms defined using lattice norms in $\mathbb{R}^2$, when the canonical order is considered.  We will give a geometric description of the compact convex sets in a two dimensional (real) asymmetric normed lattice.
More precisely, we will show that it is possible to obtain a concrete description of all the compact convex bodies (compact convex sets with a non-empty interior) in the  asymmetric space $(\mathbb R^2,q)$, where $q$ is an asymmetric lattice norm.

Finally, we will use the geometric structure of compact convex sets  to prove that every compact convex set in $(\mathbb R^{2}, q)$ is strongly $q$-compact (see definition below).

\section{Preliminaries}

The reader can find more information about asymmetric normed spaces in \cite{ale,todos,bo,cob,cobzas,bico,hous,dual,luis}). In this section we will only recall the basic notions about these spaces.

Consider a real linear space $X$ and let $\mathbb R^+$ be the set of non-negative real numbers. An \textit{asymmetric norm} $q$ on $X$ is a function  $q:X \rightarrow \mathbb R^+$ such that
\begin{enumerate}[\rm(1)]
\item $q(ax)=aq(x)$ for every $x \in X$ and $a \in \mathbb R^+$,
\item $q(x+y) \le q(x) + q(y)$, $x,y \in X$, and
\item for every $x \in X$, if $q(x)=q(-x)=0$, then $x=0$.
\end{enumerate}

The pair $(X,q)$ is called an \textit{asymmetric normed linear space}.
 The asymmetric norm defines a non-symmetric  topology on $X$ that is given by the open balls
$B_\varepsilon(x):=\{y \in X: \, q(y-x) < \epsilon\}$; this
topology is in fact the one given by the quasi-metric
$d_q(x,y):=q(y-x)$, $x,y \in X$.

For every asymmetric normed space $(X, q)$ the map $q^{s}:X\to \mathbb R^{+}$ defined by the rule $$q^s(x):= \max \{q(x),q(-x)\}, \quad x \in X,$$
is a norm that generates a topology stronger than the one generated by $q$. We will use the symbols $B_\varepsilon^q$ and $B_\varepsilon^{q^s}$ to distinguish the open unit balls of $(X,q)$ and $(X,q^s)$, respectively. In order to avoid confusion, we will say that a set is \textit{$q$-compact} ($q^{s}$-compact), if it is compact in the topology generated by $q$ ($q^{s}$).  We define the \textit{$q$-open} and \textit{$q$-closed} (\textit{$q^{s}$-open} and \textit{$q^{s}$-closed}) sets in the same way.

Since the topology is linear, the set $\theta_0:=\{x \in X:q(x)=0\}$ can be used to determine when the space is not Hausdorff. This set is also relevant for characterizing compactness, and has been systematically used in \cite{todos,luis}.

From  an abstract point of view, the knowledge  about the compact subsets of asymmetric normed linear spaces $(X,q)$ is satisfactory in  respect to their general structure (\cite{todos,new,luis}). The finite dimensional case is in fact better known (see \cite{new,luis}). In the case that the topology of $(X,q)$ is $T_1$, it is then  automatically Hausdorff --and so normed--- and its properties are thus completely known. A more interesting case results when the topology is as weak as possible, i.e. $T_0$.

The canonical example of a non-Hausdorff asymmetric normed space is when the asymmetric norm is given by a lattice norm (see the example that follows Corollary 22 in \cite{todos}).  Recall that a \textit{lattice norm} in the finite  dimensional space $\mathbb R^n$ with the coordinatewise order is a norm $\| \cdot\|$ that satisfies that $\|x\| \le \|y\|$ whenever $|x| \le |y|$, $x,y \in \mathbb R^n$. Such a norm defines an asymmetric norm $q$ by the formula
$$
q(x):= \| x \vee 0\|, \quad x \in \mathbb R^n,
$$
where $x \vee 0$ denotes the coordinatewise maximun of $x$ and $0$. An asymmetric norm defined in this way is called an \textit{asymmetric lattice norm}, and $(\mathbb R^{n}, q)$ is an \textit{asymmetric normed lattice}. A systematic study of asymmetric norms defined by means of a Banach lattice norm has been done in \cite{new} and some of the results found there are relevant for our paper. It is well known that all norms in a finite dimensional space are equivalent, and so all the asymmetric norms defined by means of this procedure  for different lattice norms are equivalent too.

These asymmetric  norms satisfy some stronger properties that are in general not true for other asymmetric norms in finite dimensional spaces. For instance, the asymmetric lattice norms are always \textit{right bounded} with constant $r=1$; that is, for every $\varepsilon >0$ and $x \in \mathbb R^n$,
$$B_\varepsilon^q(x) = B_\varepsilon^{q^s}(x) + \theta_0$$
 (see \cite[Definition 16]{luis} and   \cite[Lemma 1]{todos}).

A property that has been systematically explored in \cite{todos} for a $q$-compact set $K$ is the existence of a $q^s$-compact \textit{center} for $K$, i.e.,  a $q^s$-compact set $K_0$ such that $K_0 \subseteq K \subseteq K_0 + \theta_0$. Sets $K$ satisfying this property are called \textit{strongly $q$-compact}  (\cite{new}).

It is not difficult to see that a set $K$ is $q$-compact if and only if $K+\theta_0$ is $q$-compact. Since the topology generated by $q^{s}$ is finer than the one generated by $q$, it follows that every $q^{s}$-compact set is $q$-compact. This fact implies that every strongly $q$-compact set is $q$-compact. However, the converse implication is not always true. In fact, in \cite[Example 4.6]{new} we can find an example of a lattice norm $q$  in $\mathbb R^{2}$ and a $q$-compact  set $A\subset\mathbb R^{2}$ that is not strongly $q$-compact.
In this paper, we will show that  if we restrict our attention to the convex sets,  every $q$-compact convex set is   strongly $q$-compact.

We will use standard notation.
In the rest of this paper the letter $q$ will be used to denote an  asymmetric lattice norm in $\mathbb R^{2}$. Observe that in this case, the set $\theta_0$ coincides with the cone $(-\infty, 0]\times (-\infty, 0]$.
As usual, if $A$ is a set of $(\mathbb R^2,q)$, we write $A^o$ for the interior with respect to the asymmetric topology defined by $q$ and $\overline{A}$ for its closure. In case we want to
consider these definitions for the topology defined by $q^s$, we will write explicitly $A^{o^{q^s}}$ and $\overline{A}^{q^s}$.
Finally, we will say that a set $A\subset \mathbb R^{2}$ is a \textit{$q$-compact convex body}, if $A$ is $q$-compact, convex and its interior with respect to the asymmetric topology is non empty.
Note that $q$-convex bodies are not $q$-closed and that they may not be $q^{s}$-closed either.

\section{$q$-compact convex sets in $(\mathbb R^{2}, q)$}\label{s:ccs}

The main purpose of this section is to prove  that every $q$-compact convex set in $(\mathbb R^{2}, q)$  is strongly $q$-compact. We  will also introduce some notation and preparatory results that  will be used later to describe the geometric structure of $q$-compact convex bodies.

For the case $n=1$, let us consider the topology in $\mathbb R$ given by the asymmetric norm $| \cdot|_a:= \max\{ \cdot, 0 \}$.
We start with the following easy lemma.

\begin{lemma}  \label{lem2}

 The map $P_i:(\mathbb{R}^n,q) \to (\mathbb{R}, | ~\cdot~|_a)$ given by $P_i\big((x_1,\dots, x_n)\big)=x_i$ is continuous for every $i\in \{1,\dots, n\}$.

\end{lemma}
\begin{proof}
Recall that  $q(x)=\| x \vee 0\|$ for certain  norm  $\|\cdot \|$ in $\mathbb R^n$. Since the map $P_i:(\mathbb R^n,\|\cdot\|)\to (\mathbb R, |\cdot|)$  is continuous for every $i\in\{1,\dots ,n\}$, given $\varepsilon$ we can find $\delta>0$ such that $\|x\|<\delta$ implies that $|x_i|<\varepsilon$. Thus, if $q(x-y)<\delta$ then we get
$$\| (\max\{(x_1-y_1),0\},\dots ,\max\{(x_n-y_n),0\})\|=q(x-y)<\delta.$$ 
Now, by the choice of $\delta$ we infer that 
$$\varepsilon >|\max\{(x_i-y_i), 0\}|=\max\{(x_i-y_i), 0\}=|x_i-y_i|_a=|P_i(x)-P_i(y)|_a.$$
This last inequality implies the required result.
\end{proof}

%Proposition 4.2 in \cite{new} states that order bounded sets that contain their suprema are $q$-compact. In %a sense, this is the starting point of our analysis: we use the lattice structure of the space $(\mathbb %R^2,q)$ to find an upper bound for its asymmetric unit ball.

Is not difficult to see that a set $A\subset \mathbb R$ is $|\cdot |_a$-compact if and only if it contains it supremum.
Therefore, if  $K$ is a compact convex set  in $(\mathbb{R}^2,q)$, then the sets $P_1(K)$ and $P_2(K)$ are compact convex sets in $(\mathbb R, |\cdot |_a)$ (by Lemma~ \ref{lem2}). Then we can define the
real numbers
$$
u:=\sup \{P_1((x,y)): (x,y) \in K \} \,\,\, \text{and} \,\,\, v:=\sup \{P_2((x,y)): (x,y) \in K \} .
$$
Observe that due to the  compactness of $K$, these constants are finite
and there are elements, such as $(u,y)$ and $(x,v)$, that belong to $K$.
This means that the sets appearing below are not empty and so
there exist  $(u,\beta)$ and $(\alpha,v)$ such that
$$
\alpha:=\sup \{P_1((x,v)): (x,v) \in K \} \,\,\, \text{and}
\,\,\,\beta:=\sup \{P_2((u,y)): (u,y) \in K \}.
$$
We claim that   $(u,\beta)$ and $(\alpha,v)$ belong to
$K$. Indeed, if $(u,\beta)$ is not in $K$ then family $$\mathcal U:=\{(-\infty,t)\times \mathbb R\mid t<u\}\cup\{\mathbb R\times (-\infty, s)\mid s<\beta\}$$ covers the set $K$. By Lemma~\ref{lem2},  $(-\infty,t)\times \mathbb R=P_1^{-1}\big((-\infty, t)\big)$ is $q$-open for every $t<u$. Similarly, $\mathbb R\times (-\infty, s)=P_2^{-1}\big((-\infty, s)\big)$ is $q$-open and therefore $\mathcal U$ is an open cover of $K$. By the definition of $\beta$, we infer that $\mathcal U$ does not admit a finite subcover. This contradicts the compactness of $K$ and then we can conclude that $(u,\beta)$ belongs to $K$. Analogously, we can prove that $(\alpha,v)$ lies in $K$.

\begin{lemma} \label{lempart}
Let $K \subseteq \mathbb R^2$ be a $q$-compact convex set such that $(\alpha,v) \neq (u,\beta)$.
Let
$L_1:=\{(\alpha, y): y \le v \}$ and $L_2:= \{(x,\beta): x \le u
\}$. If $K$ has non empty interior, then the convex hull $\text{co}(L_1 \cup L_2)$ is contained in $K$.

\end{lemma}
\begin{proof} As the $q$-interior of $K$ is not empty, we can find an element
$(x_1,y_1)$ and an $\varepsilon >0$ such  that
$B_\varepsilon^q((x_1,y_1)) \subseteq K$. Since  $q$ is a lattice asymmetric norm, we find that there is a set
---unbounded with respect to the norm $q^s$--- that is included in
$B_\varepsilon((x_1,y_1))$; in fact, the set $(x_1,y_1) +
\theta_0$ is contained in this ball and so, also in $K$.

Consider the interval $[(\alpha,v),(u,\beta)]$. Since $K$ is convex, we have that this interval is included in $K$. Taking into account that it cannot be that $(\alpha,v)=(u,\beta)$, we have that $\alpha < u$ and $\beta < v$. The structure of the set $(x_1,y_1) + \theta_0$ makes it clear that there is an element $(x_0,y_0)$ in it that satisfies $(x_0,y_0) \le (s,t)$ for all
$(s,t) \in [(\alpha,v),(u,\beta)]$. Take now an element $(\alpha,t) \in L_1$, $t < v$. By the choice of $(x_0,y_0)$ we have that there is an element $(x_2,y_2) \in (x_0,y_0) + \theta_0 \subseteq (x_1,y_1) + \theta_0 \subseteq K$ such that the semiline
$$
\big\{(x_2,y_2)+ \lambda \big( (\alpha,t)-(x_2,y_2) \big): \lambda  \ge  0 \big\}
 $$
 cuts the interval  $[(\alpha,v),(u,\beta)]$ in a point $(s_0,t_0)$. Since $(\alpha,t)$ is a convex combination of $(x_2,y_2)$ and $(s_0,t_0)$ we have that $L_1$ is included in $K$.

Repeating the same argument for $L_2$, we obtain that it is also included in $K$, and since $K$ is convex we
obtain that $\text{co}( L_1 \cup L_2) \subseteq K$.

\end{proof}

Throughout the rest of this section, $K$ will always denote a $q$-compact convex subset in $(\mathbb R^{2}, q)$. In order to complete the geometric
information about the structure of $K$, let us define the following sets.
$$
\Delta_K:= \text{co}\big\{ (\alpha,\beta), (\alpha,v), (u,\beta) \big\},
$$
$$
S_K:=K \cap \big( \mathbb{R} \times [\beta,+\infty) \big) \cap \big( [\alpha,
+\infty) \times \mathbb R \big)
$$
and

$$R_K:=K\cap H,$$
where $H$ is the closed half upper plane  determined by the line containing  the segment $[(u,\beta), (\alpha, v)]$.
Observe that $R_K$ and $S_K$ are both convex sets contained in the rectangle $[\alpha, u]\times [\beta, v]$ and therefore, they are $q^{s}$-bounded.

Finally,  let us consider the set
\begin{equation*}
F_K:=\begin{cases} \partial_{q^{s}}\big(\overline{R_K}^{q^s}\big) \setminus ((u,\beta), (\alpha, v)),& \text{if }R_K\neq [(u,\beta), (\alpha, v)]\\
[(u,\beta), (\alpha, v)], &\text{if }R_K= [(u,\beta), (\alpha, v)]
\end{cases}
\end{equation*}
where $\overline{R_K}^{q^s}$ is the $q^s$-closure of $R_K$ and $\partial_{q^{s}}$ denotes the boundary with respect to the topology determined by $q^{s}$.
Since $R_K$ is convex, the set $F_K$ is a simple arc.
Furthermore, we have the following inclusions
 \begin{equation*}\label{e: contenciones}
 R_K\subset S_K\subset K.
\end{equation*}
 Additionally, since $K$ is convex and $F_K$ is an arc, it follows from  the definition of $(\alpha,v)$ and $(u,\beta)$,  that
$F_K$
cannot contain any horizontal or vertical non-trivial segment.

  If $K$ has non empty interior, by Lemma~\ref{lempart} we also have that  $\Delta_K \subseteq S_K$.

\begin{lemma} \label{inclufron}
 $F_K \subseteq K$.
\end{lemma}
\begin{proof}
Let $(x_1,y_1) \in F_K \subseteq \overline{R_K}^{q^s}$ and assume that
it does not belong to $K$. Let $U_t:=(- \infty, x_1 -t) \times \mathbb
R$ and $V_t:=\mathbb R \times (-\infty, y_1 -t)$, $t >0$. Both of
them are $q$-open. Since $F_K$ does not contain either an horizontal
or a vertical segment, we have that the sets $\mathcal U:= \{U_t,
V_t \}_{t>0}$ constitute a $q$-open cover for $K$, and so there are $t_1, t_2$ such that $K
\subseteq U_{t_1} \cup V_{t_2}$.

Moreover,  $(x_1,y_1) \in (x_1-t_1, +\infty) \times
(y_1-t_2, + \infty)$, which is a $q^s$-open set that does not intersect $K$.
This is a contradiction with the fact that $(x_1,y_1)$ is in $\overline{R_K}^{q^s}\subset \overline{K}^{q^s}$.
\end{proof}

\begin{lemma} \label{S1compact}
The set $R_K$ is a
$q^s$-compact  convex set. Additionally, if $K$ has non empty interior, then $S_K$ is $q^{s}$-compact too.
\end{lemma}

\begin{proof}

Observe that $R_K$ coincides with the convex hull $\text{co}(F_K)$.  Since $F_K$ is $q^s$-compact,  we can use a well-known result (see, e.g., \cite[Corollary 5.33]{Aliprantis Border}) to conclude that $R_K$ is $q^s$-compact too.
In the other hand, if $K$ has non empty interior, then $S_K=\text{co}(F_K\cup \{(\alpha,\beta)\})$. Using again \cite[Corollary 3.3]{Aliprantis Border}, we infer that $S_K$ is $q^s$-compact, as desired.

\end{proof}

\begin{theorem}\label{t:strong compact}
Every $q$-compact convex set $K$ in $(\mathbb R^{2},q)$ is strongly $q$-compact.
\end{theorem}

\begin{proof}
By Lemma~\ref{S1compact}, the set $R_K$ is $q^{s}$-compact and satisfies $R_K\subset K$.  Let us prove that $K\subset R_K+\theta_0$.
First observe that
$$K\setminus R_K\subset \Delta_K\cup\{(\alpha, v)+\theta_0, (u,\beta)+\theta_0\}.$$
On the other hand, since the points $(\alpha, v)$ and $(u,\beta)$ belong to $R_K$, it follows that $(\alpha, v)+\theta_0\subset R_K+\theta_0$ and $(u, \beta)+\theta_0\subset R_K+\theta_0$.
It only rests to prove that $\Delta_K\subset R_{K}+\theta_0$, but this is obvious since  $R_{K}+\theta_0$ is convex and the point
$(\alpha,\beta)\in (\alpha,v)+\theta_0\subset R_K+\theta_0$. Thus
$$\Delta_K=\text{co}\{(\alpha,\beta), (\alpha,v), (u,\beta)\}\subset R_K+\theta_0.$$
Now the proof is complete.
\end{proof}

\begin{remark}

In   \cite[Theorem 4.7]{new} it was proved that for the $q^s$-closed subsets of $\mathbb R^2$, the notions of $q$-compactness and strong $q$-compactness coincide. However, as mentioned in Section 2, even in $\mathbb R^{2}$ there are examples of (non $q^{s}$-closed) $q$-compact sets that are not strongly $q$-compact.
Since each strongly $q$-compact set is $q$-compact (see \cite[Proposition 11]{todos}), our result implies that for the case of convex sets this result is also true, even if it is not $q^s$-closed.
\end{remark}

\section{The structure of $q$-convex bodies}

For the rest of the paper, the letter  $K$ will always denote  a $q$-compact convex body.
Let us also recall that, if $x,y \in \mathbb R^2$, $\{x,y]$ means either the closed interval $[x,y]$ or the left open interval $(x,y]$.

\begin{remark} \label{rem1point} Let $K\subset \mathbb R^{2}$ be a $q$-compact convex body. Consider the corresponding points $(\alpha,v)$ and $(u,\beta)$ computed  just as in Section~\ref{s:ccs}.
Note that $\beta=v$ if and only if  $\alpha=u$. In this case, we
have that $F_K= \{(u,v) \} \subseteq K$, which implies
necessarily that
$$
K \subseteq (u,v) + \theta_0  \subseteq \overline{K}^{q^s}.
$$
For instance, a set such as $((x_0,y_0) + (\theta_0)^o)\cup\{(x_0,y_0)\}$, where $(\theta_0)^o$ denotes the $q$-interior of
$\theta_0$, is a $q$-compact set; this is easy to verify. This implies that for the case $\beta=v$, the $q$-compact convex body $K$ is necessarily
$$
K=  \{(u,t_0),(u,v)] \cup \{(s_0,v),(u,v)] \cup ((u,v)+(\theta_0)^o),
$$
for some $-\infty \le t_0  \le v$ and $-\infty \le s_0 \le u$.
Such a set may not be homeomorphic to the asymmetric $q^s$-closed unit ball $\overline{B_1^q(0)}^{q^s}$ in
$\mathbb R^2$, since in this case the natural
identification of the boundaries from $K$ to $\overline{B_1^q(0)}^{q^s}$ may not be
surjective (for example if $t_0=v$ and $s_0=u$).
\end{remark}

Everything is now prepared to prove the main result of this section.

\begin{theorem} \label{main}
Let $K$ be a subset of the asymmetric Euclidean space $(\mathbb R^2,q)$. The following statements are equivalent.
\vspace{3mm}
\begin{itemize}
\item[(i)]
$K$ is a $q$-compact convex body.
\vspace{3mm}
\item[(ii)]

\begin{itemize}
\item[(1)] Either there is a point $(u,v) \in \mathbb R^2$ such that
$$
K=  \{(u,t_0),(u,v)] \cup \{(s_0,v),(u,v)] \cup \big( (u,v)+ (\theta_0)^o \big)
$$

for some $-\infty \le t_0  \le v$ and $-\infty \le s_0 \le u$,
\vspace{3mm}
\item[(2)]
or there are real scalars $s_0\leq \alpha < u$, $t_0\leq \beta < v$ and a $q^s$-compact convex set $K_0$ satisfying that
$$
\text{co} \{(\alpha,v),(\alpha,\beta),(u,\beta) \} \subseteq K_0 \subseteq \text{co} \{(\alpha,v),(\alpha,\beta),(u,\beta), (u,v) \},
$$
$$K_0\subset K\subset K_0+\theta_0$$
and such that
$$
K= K_0 \cup \big( (\alpha,v) + (\theta_0)^o \big) \cup \big( (u,\beta) + (\theta_0)^o \big) \cup \{(u,t_0),(u,\beta)] \cup \{(s_0,v),(\alpha,v)].
$$

\end{itemize}

\end{itemize}
\end{theorem}

\begin{proof}
(i) $\Rightarrow$ (ii)
Take a $q$-compact convex body $K$ and compute the numbers $(\alpha,v)$ and $(u,\beta)$ as explained in Section~\ref{s:ccs}.
 By Theorem~\ref{t:strong compact}, we have that
$$S_K\subset K\subset R_K+\theta_0\subset S_K+\theta_0.$$
Moreover, the proof of Lemma~\ref{lempart}
gives that
$$
K_1:=S_K \cup
\big( (\alpha,v) + (\theta_0)^o \big) \cup \big( (u,\beta) + (\theta_0)^o \big) \subseteq K\subset S_K+\theta_0.
$$
Defining $s_0=\inf\{s\mid (s,v)\in K\}$ and $t_0=\inf\{t\mid (u,t)\in K\}$  (allowing $s_0=-\infty$ and $t_0=-\infty$, if that is the case), we then have:
\begin{equation}\label{e:Kigual K1}
K=K_1\cup \{(s_0, v), (\alpha, v)]\cup\{(u,t_0), (u,\beta)].
\end{equation}
Now, if $(\alpha, v)=(u,\beta)$, then $K_1=(u,v)+(\theta_0)^o$ and therefore:
$$K=\{(s_0, v), (u, v)]\cup\{(u,t_0), (u,v)]\cup ((u,v)+(\theta_0)^o).$$
If is not the case, let $K_0=S_K$.  By Lemma~\ref{S1compact}, $S_K$ is $q^{s}$-compact and obviously  satisfies
 $$\text{co} \{(\alpha,v),(\alpha,\beta),(u,\beta) \} \subseteq S_K \subseteq \text{co} \{(\alpha,v),(\alpha,\beta),(u,\beta), (u,v) \}.$$
 Finally, just use the definition of $K_1$ in combination with equality (\ref{e:Kigual K1}) to conclude that 
$$ K=S_K \cup \big( (\alpha,v) + (\theta_0)^o \big) \cup \big( (u,\beta) + (\theta_0)^o \big) \cup \{(u,t_0),(u,\beta)] \cup \{(s_0,v),(\alpha,v)].$$

This proves the first implication.

(ii) $\Rightarrow$  (i)  It follows from the structure of $K$ that it has non empty interior.
If the set $K$ is as in (1), it is obviously $q$-compact since each open set containing the point $(u,v)$ contains the whole set $K$.
If the set $K$ is as in (2), then it is strongly $q$-compact and thus also $q$-compact. Now the proof is complete.
\end{proof}

By Theorem~\ref{t:convexos dimension 2} we know that the $q^{s}$-closure of a $q$-compact convex body $K\subset \mathbb R^{2}$ is $q^{s}$-homeomorphic to $\overline{B^{q}_{1}(0)}^{q^s}$.
Using the geometric description of the $q$-convex bodies of the asymmetric space $(\mathbb R^{2},q)$ given in Theorem \ref{main}   we get the following

\begin{corollary}
Let $K$ be a $q$-compact convex body. It is  then $q^s$-homeomorphic to a convex set $K'
\subseteq \overline{B^{q}_{1}(0)}^{q^s}$  that can be written as
$$
K' = B^{q}_{1}(0)\cup C
$$
where $C$ is a connected subset of the $q^s$-boundary of $\overline{B^{q}_{1}(0)}^{q^s}$.
\end{corollary}

\textbf{Final Remark.}

 It is well known that convex sets in $\mathbb R^{2}$  satisfy properties that are not generally true for convex sets in higher dimensions. The same happens in the asymmetric case. For instance, Theorem~\ref{t:strong compact}, is no longer true in $n$-dimensional asymmetric normed lattices if $n>2$. An example of this situation is explained below.

\begin{example}\label{e:dimension 3}
Consider the asymmetric normed lattice  $(\mathbb R^{3}, q)$, where $q:\mathbb R^{3}\to[0,\infty)$ is the asymmetric lattice norm defined by the rule:
$$q(x)=\max\{\max\{x_i, 0\}\mid i=1,2,3\}\quad x=(x_1,x_2,x_3)\in \mathbb R^{3}.$$
Let $K=\text{co} (A\cup\{(0,0,0),(0,1,1)\})$ where $A$ is the set defined as 
$$A=\{(x_1,0,x_3)\mid x_1^{2}+x^{2}_3=1,\,x_1\in (0,1],\,x_3\geq 0 \}.$$
For any $q$-open cover $\mathcal U$ of $K$, there exists an element $U\in \mathcal U$ such that $(0,1,1)\in U$.  By \cite[Lemma 4]{luis}, this implies that 
$$(0,1,1)+\theta_0\subset U+\theta_0=U,$$
and therefore the cover $\mathcal U$ is a $q$-open (and $q^{s}$-open) cover  for $\overline {K}^{q^s}$ which  is $q^{s}$-compact. Thus, we can extract a finite subcover $\mathcal{V}\subset\mathcal U$ for   $\overline {K}^{q^s}$. This cover $\mathcal V$ is a finite subcover for $K$ too, and then we can conclude that $K$ is $q$-compact.

To finish this example, let us note that $K$ is not strongly $q$-compact. For this, simply observe that any set $K_0$ satisfying $K_0\subset K\subset K_0+\theta_0$, must contain the set $A$. If $K_0$ is additionally $q^{s}$-compact, then it is $q^{s}$-closed too and then
$$\overline{A}^{q^s}\subset \overline{K_0}^{q^s}=K_0\subset K,$$
which is impossible. 
\end{example}

Since the situation in higher dimensions changes drastically from the $2$-dimensional case,  we decided to work this case separately in this note. 
The geometric structure of finite dimensional $q$-compact convex set in an asymmetric normed space, can be characterized in terms of its extreme points and the convex hull of these extreme points somehow determines the compactness of the set. To see these and other results,  the reader can consult \cite{Jonard Sanchez}.

\textit{The authors are very grateful to the referees since their suggestions, remarks and comments have permitted a substantial improvement of the first version of the paper.}

%\section{Bibliografia}

\vspace{2cm}

\noindent[Natalia Jonard-P\'erez]  Departamento de Matem\'aticas,
Facultad de Matem\'aticas,
Campus Espinardo,
30100 Murcia, Spain, e-mail: nataliajonard@gmail.com

\vspace{1cm}

\medskip

\noindent[Enrique A. S\'anchez P\'erez] Instituto Universitario de Matem\'{a}tica Pura y Aplicada, Universitat Polit\`ecnica de Val\`encia, Camino de Vera s/n, 46022 Valencia, Spain, e-mail: easancpe@mat.upv.es


\begin{thebibliography}{99}





\bibitem{ale}  C. Alegre, J. Ferrer, and V.  Gregori, \textit{On the
Hahn-Banach theorem in certain linear quasi-uniform structures},
Acta Math. Hungar. 82 (1999), 315-320.

\bibitem{todos}  C. Alegre,  I. Ferrando, L. M.  Garc\'{\i}a Raffi,
and E. A. S\'anchez P\'erez,  \textit{Compactness in asymmetric
normed spaces}, Topology Appl. 155 (2008), 527-539.

\bibitem{Aliprantis Border} C. D. Aliprantis and K. Border, \emph{Infinite Dimensional Analysis. A Hitchiker's Guide}, 3th, edition, Springer, Berlin Heidelberg, 2006.

\bibitem{bespel}
C. Bessaga, and A. Pe{\l}czy\'{n}ski,
\textit{Selected topics in infinite-dimensional topology}.
PWN - Polish Scientific Publishers. Warszawa. 1975.


\bibitem{bo} P. A. Borodin,  \textit{The Banach-Mazur theorem for
spaces with asymmetric norm and its applications in convex
analysis}, Mathematical Notes, 69  (3) (2001), 298-305.

\bibitem{cob} S.
Cobza\c{s}, \textit{Separation of convex sets and best
approximation in spaces with asymmetric norm}, Quaest. Math. 27, 3
(2004), 275-296.


\bibitem{cobzas}  S. Cobza\c{s}, \textit{Functional analysis in asymmetric normed spaces}, Birkh\"auser. Basel. 2013.


\bibitem{new}
  J. J. Conradie  and M. D.
Mabula,  \textit{
Completeness, precompactness and compactness in
finite-dimensional asymmetrically normed lattices},
Topology Appl. 160 (2013), 2012-2024.


\bibitem{gar4}  L. M. Garc\'{\i }a-Raffi, S. Romaguera,  and
 E. A. S\'{a}nchez-P\'{e}rez, \textit{Sequence spaces and asymmetric
norms in the theory of computational complexity}, Math. Comp.
Model. 36 (2002), 1-11.

\bibitem{bico} L. M. Garc\'{\i }a Raffi, S. Romaguera,  and E. A.
S\'{a}nchez P\'{e}rez,  \textit{The bicompletion of an
asymmetric normed linear space}, Acta Math. Hungar. 97 (3) (2002),
183-191.


\bibitem{hous}  L. M. Garc\'{\i }a Raffi, S. Romaguera,  and E. A. S\'{a}nchez
P\'{e}rez,  \textit{On Hausdorff asymmetric normed linear
spaces}, Houston J. Math., 29 (2003), 717-728.


\bibitem{dual}  L. M. Garc\'{\i }a Raffi, S. Romaguera,  and
E. A. S\'{a}nchez-P\'{e}rez, \textit{The dual space of an
asymmetric normed linear space}, Quaest. Math. 26 (2003),
83-96.

\bibitem{luis} L. M.
Garc\'{\i }a Raffi,  \textit{Compactness and finite dimension
in asymmetric normed linear spaces}, Topology Appl. 153 (2005),
844-853.


\bibitem{Jonard Sanchez} N. Jonard-P\'erez and E. A. S\'anchez-P\'erez, \emph{Extreme points and geometric aspects of compact convex sets in asymmetric normed spaces}, http://arxiv.org/abs/1404.0500 (preprint).

\bibitem{kot}
G. K\"{o}the, \textit{Topological Vector Spaces I}, Springer.
Berlin. 1969.

\bibitem{lint}
J. Lindenstrauss and L. Tzafriri, \textit{Classical Banach Spaces
II}, Springer. Berlin. 1996.

\bibitem{ro2} S. Romaguera and M. Schellekens,  \textit{Quasi-metric properties on complexity
spaces}, Topology Appl. 98 (1999), 311-322.

\bibitem{ro1} S. Romaguera and  M. Sanchis,  \textit{Semi-Lipschitz functions
and best approximation in quasi-metric spaces}, J. Approx. Th. 102
(2000), 292-301.



\bibitem{ro3} S. Romaguera  and M. Sanchis,  \textit{Duality and quasi-normability for complexity
spaces}, Appl. Gen. Topology 3 (2002), 91-112.


\end{thebibliography}
\end{document}